\documentclass[11pt,reqno]{amsart}

\usepackage{amsmath,amsthm,amssymb,amsfonts}
\usepackage{esint, tikz}
\usetikzlibrary{matrix}
\usepackage{enumerate}
\usepackage[paperwidth=8.26in,paperheight=11.69in,
left=1.125in, right=1.125in, bottom=1.25in]{geometry}

\usepackage{ifpdf}
\ifpdf
\usepackage{thmtools}

\usepackage{lmodern}\usepackage{esint}
\declaretheoremstyle[
 headfont=\normalfont\bfseries,
 headindent=\normalparindent,
 bodyfont=\em,
 spaceabove=8pt,
 spacebelow=8pt
]{thm}
\declaretheoremstyle[
 headfont=\normalfont\em,
 headindent=\normalparindent,
 spaceabove=8pt,
 spacebelow=8pt
]{remark}
\declaretheoremstyle[
 headfont=\normalfont\bfseries,
 headindent=\normalparindent,
 spaceabove=8pt,
 spacebelow=8pt
]{example}
\declaretheoremstyle[
 headfont=\normalfont\bfseries,
 headindent=\normalparindent,
 spaceabove=8pt,
 spacebelow=8pt
]{definition}

\declaretheorem[name=Theorem,style=thm,numberwithin=section,
]{thm}
\newtheorem*{thm*}{Theorem}
\declaretheorem[name=Proposition,style=thm,sibling=thm]{prop}
\declaretheorem[name=Lemma,style=thm,sibling=thm]{lem}
\declaretheorem[name=Corollary,style=thm,sibling=thm]{cor}

\usepackage{cleveref}
\crefname{thm}{Theorem}{Theorems}
\crefname{prop}{Proposition}{Propositions}
\crefname{lem}{Lemma}{Lemmas}
\crefname{cor}{Corollary}{Corollaries}
\crefname{example}{Example}{Examples}
\crefname{defn}{Definition}{Definitions}
\crefname{rem}{Remark}{Remarks}
\crefname{enumi}{}{}
\crefname{enumii}{}{}
\crefname{enumiii}{}{}
\crefname{equation}{}{}

\usepackage{enumitem}
\setlist{leftmargin=2.5em}
\setlist[enumerate,1]{label=(\roman*),ref=(\roman*)}

\numberwithin{equation}{section}

\newcommand{\be}{\begin{equation}}
\newcommand{\ee}{\end{equation}}
\newcommand{\bel}{\begin{equation}}
\newcommand{\eel}{\end{equation}}
\newcommand{\bee}{\begin{eqnarray*}}
\newcommand{\eee}{\end{eqnarray*}}

\newcommand{\RR}{\mathbb{R}}
\newcommand{\CC}{\mathbb{C}}

\newcommand{\boxop}{\Box_{h,\mu}}
\newcommand{\dbarstar}{\dbar^*_{h,\mu}}
\newcommand{\energ}{\mathcal{E}_{h,\mu}}

\DeclareMathOperator{\trace}{tr}
\DeclareMathOperator{\vol}{Vol}
\newcommand{\dvol}{d\!\vol}

\DeclareMathOperator{\dom}{dom}

\DeclareMathOperator{\Lip}{Lip}
\newcommand{\lip}{\Lip(M,h)\cap L^\infty(M)}

\newcommand{\overbar}[1]{\mkern 1mu\overline{\mkern-1mu#1\mkern-1mu}\mkern 1mu} % shorter \overline; from http://tex.stackexchange.com/questions/22100/the-bar-and-overline-commands

\newcommand{\ol}[1]{\overline{#1}}

\newcommand{\Cplx}{\mathbb{C}}
\newcommand{\Cplxi}{i}

\renewcommand{\Re}{\operatorname{Re}} % real part
\renewcommand{\leq}{\leqslant}
\renewcommand{\geq}{\geqslant}

\renewcommand{\dim}[1][]{\operatorname{dim}_{#1}}

\newcommand{\dbar}{{\smash{\overbar{\partial}}}} % dbar operator

\usepackage{mathtools}

\newcommand{\bergman}[1][2]{A^{#1}} % Bergman space

\usepackage[
activate={true,nocompatibility},
final,
tracking=true,
kerning=true,
spacing=true,
factor=1100,
stretch=10,
shrink=10
]{microtype}
\SetTracking{encoding={*}, shape=sc}{0} % Fix for small caps letter spacing
\microtypecontext{spacing=nonfrench}

\begin{document}
\title[]{Exponential decay of Bergman kernels on complete Hermitian manifolds with Ricci curvature\\ bounded from below}
\author{Franz Berger}
\author{Gian Maria Dall'Ara}
\author{Duong Ngoc Son}
\address{Fakult\"{a}t f\"{u}r Mathematik, Universit\"{a}t Wien, Oskar-Morgenstern-Platz 1, 1090 Wien, Austria }
\email{franz.berger2@univie.ac.at}
\email{gianmaria.dallara@univie.ac.at}
\email{son.duong@univie.ac.at}

\date{\today}
\thanks{The first and second-named authors were supported by the Austrian Science Fund (FWF) project P28154.}
\thanks{The third-named author was supported by the Austrian Science Fund (FWF) project I01776.}

\begin{abstract}
Given a smooth positive measure $\mu$ on a complete Hermitian manifold with Ricci curvature bounded from below, we prove a pointwise Agmon-type bound for the corresponding Bergman kernel, under rather general conditions involving the coercivity of an associated complex Laplacian on $(0,1)$-forms.  Thanks to an appropriate version of the Bochner--Kodaira--Nakano basic identity, we can give explicit geometric sufficient conditions for such coercivity to hold. 

Our results extend several known bounds in the literature to the case in which the manifold is neither assumed to be K\"ahler nor of ``bounded geometry''. The key ingredients of our proof are a localization formula for the complex Laplacian (of the kind used in the theory of Schr\"odinger operators) and a mean value inequality for subsolutions of the heat equation on Riemannian manifolds due to Li, Schoen, and Tam.

We also show in an appendix that the ``twisted basic identities'' of, e.g., \cite{mcneal--varolin} are standard basic identities with respect to conformally K\"ahler metrics.
\end{abstract}

\maketitle

%\tableofcontents % The table of contents means to be just temporary, although we can keep it in the final version if we like.

\section{Introduction}\label{sec:intro}

\subsection{The problem and previous results}\label{sec:problem} Bergman spaces of holomorphic functions and related Bergman kernels are classical objects of complex analysis and geometry (see, e.g., \cite{bergman,krantz,ma--marinescu/book} and the references therein). If $M$ is a complex manifold and $\mu$ a positive Borel measure on $M$,  the \emph{Bergman space} $\bergman(M,\mu)$ is the linear space of square-integrable holomorphic functions on $M$, i.e.,
\begin{equation}
	\bergman(M,\mu) \coloneqq \bigg\{f \colon M \to \Cplx : f \text{ is holomorphic and } \int_M |f|^2\,d\mu < \infty \bigg\}.
\end{equation}
Under mild assumptions on $\mu$, $A^2(M,\mu)$ is closed in $L^2(M,\mu)$ and actually a reproducing kernel Hilbert space (see Section \ref{sec:bergman}). The \emph{Bergman kernel} $K_\mu \colon M \times M \to \Cplx$ is defined by the relation 
\begin{equation}
	B_\mu f(p) = \int_M K_\mu(p,q) f(q)\,d\mu(q), \qquad f \in L^2(M,\mu),\quad  p \in M,
\end{equation} 
where $B_\mu$ is the orthogonal projection of $L^2(M,\mu)$ onto $A^2(M,\mu)$.

It has been shown (see, e.g., \cite{christ,delin,marzo--ortega-cerda,lindholm,ma--marinescu,schuster--varolin,dallara}) that, under various assumptions on $\mu$, one can find a Hermitian metric $h$ such that the following Agmon-type pointwise decay estimate holds:
\bel\label{exp-decay}
	|K_\mu(p,q)| e^{-\psi(p)-\psi(q)}\leq C\frac{e^{-\gamma d(p,q)}}{\sqrt{\vol(p,1)\vol(q,1)}} \qquad (p,q\in M).
\eel
Here $d(p,q)$ is the Riemannian distance between $p$ and $q$, $\vol$ is the Riemannian volume, $\vol(p,1)$ is the volume of the ball centered at $p$ and of radius $1$, and $\psi$ is determined by the relation $\mu = e^{-2\psi} \vol$. The positive constants $C$ and $\gamma$ do not depend on $p$ and $q$. We point out that $U_\psi f\coloneqq e^{-\psi}f$ is a unitary isomorphism of $L^2(M,\mu)$ onto $L^2(M,\vol)$, and $U_\psi\circ B_\mu\circ U_\psi^{-1}$ is an orthogonal projector on $L^2(M,\vol)$. The left-hand side of \cref{exp-decay} is thus the modulus of the integral kernel of this projector, and the estimate shows that this kernel exhibits an off-diagonal exponential decay, which can be neatly expressed in terms of the metric $h$. Estimates of the form \cref{exp-decay} have had numerous applications in complex analysis and geometry (see, e.g.,~\cite{ma--marinescu} and \cite{lu--zelditch} and the references therein). 

Typically, assumptions for \cref{exp-decay} to hold can be formulated as conditions on the ``curvature form'' $F^\mu$ of $\mu$, which is defined as follows: in local holomorphic coordinates, one writes $d\mu=\Cplxi e^{-2\varphi}\,dz^1\wedge dz^{\ol 1}\wedge\dotsb\wedge dz^n\wedge dz^{\ol n}$, where $\varphi$ is smooth and real-valued. Then one can easily check that $F^{\mu} \coloneqq i\partial\dbar\varphi$ is a real $(1,1)$-form which does not depend on the choice of the coordinates. Thus, $F^{\mu}$ is globally defined and we shall call it the \emph{curvature form} of~$\mu$.

We now proceed to describe some of the aforementioned results in a little more detail.
\begin{enumerate}[label={(\arabic*)}]
\item\label{item:christ} In the one-dimensional case $M=\CC$, $\mu=e^{-2\psi}\lambda$, where $\lambda$ is Lebesgue measure and $\psi$ is subharmonic, the curvature form $F^\mu$ may be identified with (a multiple of) the measure $\Delta\psi\,\lambda$, where $\Delta$ is the usual Laplacian. It was shown by Christ~\cite{christ} (but see also \cite{marzo--ortega-cerda}) that if $F^\mu$ is doubling and satisfies 
\be
	\inf_{z\in\CC} F^\mu(D(z,1))>0,  \quad \text{where}\ D(z, r)\coloneqq \{z\in \Cplx \colon |z|< r\},
\ee 
then \cref{exp-decay} holds with respect to the metric $h=\rho^{-2}|dz|^2$, with 
\be
	\rho(z)\coloneqq\sup\big\{r>0\colon F^\mu(D(z,r))\leq 1\big\}.
\ee 

\item\label{item:delin} In \cite{delin}, Delin considers $M=\CC^n$ and $\mu=e^{-2\psi}\lambda$, where $\psi$ is strictly plurisubharmonic, and proves an estimate that takes the form \cref{exp-decay} with $h$ the K\"ahler metric $i\partial\dbar\psi$, at least when certain quantitative assumptions on $F^\mu$ are made. These conditions are not explicitly discussed by Delin (see the comment after the statement of Theorem~2 of \cite{delin}), but it is certainly sufficient that\bel\label{lindholm}
ic\partial\dbar |z|^2 \leq F^\mu\leq iC\partial\dbar |z|^2
\eel holds for some $0<c$ and $C<+\infty$, as shown in \cite[Proposition 9]{lindholm}.

\item\label{item:dallara} In \cite{dallara}, the second-named author deals with $M=\CC^n$ and $\mu=e^{-2\psi}\lambda$, where $\psi$ is only assumed to be \emph{weakly} plurisubharmonic. More precisely, if $\Delta\psi$ is in the reverse-H\"older class $RH_{\infty}$, and
\be\label{e:dallara1}
	\inf_{z\in\CC^n} \sup_{|w-z|<1}\Delta\psi(w)>0,
\ee
then estimate \cref{exp-decay} holds under the condition (for some $c>0$)
\begin{equation}\label{e:dallara15}
	F^\mu\geq ic\,\Delta\psi\,\partial\dbar |z|^2.
\end{equation}
In this case the metric is $h=\rho^{-2}|dz|^2$, where
\be
	\rho(z)\coloneqq\sup\big\{r>0\colon \textstyle{\sup_{|w-z|<r}}\Delta\psi(w)\leq r^{-2}\big\}.
\ee
The condition \cref{e:dallara15} amounts to the uniform comparability of the eigenvalues of the complex Hessian $(\partial_{z_j}\partial_{\ol z_k}\psi)_{j,k}$.
Notice that \cref{lindholm} implies \cref{e:dallara1} and \cref{e:dallara15}.

\item\label{item:schuster-varolin} In \cite{schuster--varolin}, Schuster and Varolin take $M$ to be the unit ball $\mathbb{B}\subseteq\CC^n$ endowed with the \emph{Bergman metric} \(\omega\coloneqq-\frac{i}{2}\partial\dbar\log(1-|z|^2)\), and prove \cref{exp-decay} for measures $\mu=e^{-2\psi}\omega^n/n!$, under the condition
\begin{equation}\label{e:SV14condition}
	(n+\sigma)\omega\leq i\partial\dbar\psi\leq C\omega,
\end{equation}
where $\sigma>\frac{1}{2}$. One can see that $F_\mu=i\partial\dbar\psi-(n+1)\omega$ in this case (by \cref{Fmu-formula} below and the fact that $\omega$ is K\"ahler--Einstein with Ricci curvature $\Theta=-2(n+1)\omega$), and hence \cref{e:SV14condition} is equivalent to
\be
	c\omega \leq F_\mu\leq C\omega
\ee 
for some $c>-1/2$ and $C<+\infty$. 

The result was generalized by Asserda \cite{asserda} to K\"ahler manifolds satisfying a certain bounded geometry assumption.

\item\label{item:ma-marinescu} In \cite{ma--marinescu}, Ma and Marinescu prove a pointwise $C^k$ estimate for the Bergman kernels in the more general setting of Hermitian line bundles over symplectic manifolds (satisfying appropriate compatibility conditions). Specializing to the present situation, Theorem~1 in that paper requires in particular that the Hermitian manifold $(M,h)$ has ``bounded geometry'', and that the measure $\mu=e^{-2\psi}\vol$ is such that\be
c\omega_h\leq i\partial\dbar\psi\leq C\omega_h
\ee (where $\vol$ is the Riemannian volume and $\omega_h$ the fundamental form) for $c>0$ and $C<+\infty$. Then, if $k>0$ is large enough, the measure $\mu^{(k)}=e^{-2k^2\psi}\vol$ satisfies \bel\label{ma-marinescu}
|K_{\mu^{(k)}}(p,q)| e^{-k^2\psi(p)-k^2\psi(q)}\leq Ck^{2n}e^{-\gamma kd(p,q)}, \qquad (p,q\in M)
\eel 
with $C$ independent of $p$, $q$, and $k$. Notice that the absence of the volume factors in \cref{ma-marinescu} is due to the bounded geometry assumption. In fact, if the volumes of balls with a fixed positive radius is bounded away from zero (which is the case if the sectional curvature is bounded from above (by \cite[Theorem~3.101]{gallot--hulin--lafontaine}) then the volume factors can be absorbed into the constant~$C$.
\end{enumerate}

These results, despite being of the same nature, present two different points of view on the problem of establishing exponential decay of Bergman kernels: \cref{item:christ,item:delin,item:dallara} start with a measure $\mu$ and construct a metric $h$ with respect to which the exponential decay \cref{exp-decay} holds, while \cref{item:schuster-varolin,item:ma-marinescu} start with a Hermitian manifold and look for conditions on the density of $\mu$ with respect to the Riemannian volume that are sufficient for \cref{exp-decay} to hold. Moreover, in \cref{item:christ,item:delin,item:dallara} a natural candidate for $h$ is the K\"ahler metric with fundamental form $F^\mu$, but the latter form need not be positive, and in fact \cref{item:christ,item:dallara} consider a sort of regularization of $F^\mu$ and the resulting metric is typically \emph{non-K\"ahler}. 

\subsection{Our results}

To state our results, we shall need to recall and fix some more notation. Let $(M,h)$ be a complete Hermitian manifold with a Hermitian metric~$h$:
\be 
	h=h_{j\ol k} \, dz^j \otimes dz^{\ol k}.
\ee
The associated $(1,1)$-form \(\omega_h\coloneqq ih_{j\ol k}\, dz^j\wedge dz^{\ol k}\) is called the \emph{fundamental form}. As usual, we refer to both $h$ and $\omega_h$ as a metric on $M$. As is well-known, the torsion tensor $T$ of the Chern connection is non-trivial if and only if the metric is K\"ahler: locally, $T$ has components
\be 
	T_{jk}^{\ell}
	=
	\ol{T_{\ol j\ol k}^{\ol\ell}}
	=
	h^{\ell \overline{m}}\left(\partial_j h_{k\overline{m}} - \partial_k h_{j\overline{m}}\right).
\ee 
We shall deal with the \emph{torsion $1$-form}, obtained by taking the trace of the torsion:
\bel\label{torsion-1-form}
\theta = T^k_{jk}\, dz^j + T^{\ol k}_{\ol j\ol k}\, dz^{\ol j},
\eel 
and the \emph{torsion $(1,1)$-form} $T\circ \ol T$  defined by% for which we borrow the notation of \cite[p. 5171]{liu--yang}.
\bel\label{torsion-squared}
T\circ \ol T\coloneqq h^{a\ol\ell}h^{b\ol m}h_{\ol q j} h_{p\ol k}\,  T^{p}_{a b}\, \ol{T^{q}_{\ell m}}\, dz^j\wedge dz^{\ol k}.
\eel
The Riemannian metric $g\coloneqq2\Re h$ induces a distance $d_h$ and a volume $\vol$. We denote by $\vol(p, R)$ the volume of the metric ball $B(p,R) \coloneqq \{q \in M : d(p,q) \leq R\}$ of radius $R$ centered at~$p$.\newline
%By completeness, $B(p,R)$ equals the set of points reached by geodesics starting at $p$ after time at most $R$, whence we shall refer to these sets as \emph{geodesic balls}.
Since a Hermitian metric $h$ induces inner products for tensors of all ranks, we can consider the space $L^2_{0,q}(M,h,\mu)$ of square-integrable $(0,q)$-forms on $M$, with inner product given by 
\be
	(u,v) \mapsto \int_M\langle u,v\rangle_h\, d\mu.
\ee
We denote by $\dbarstar$ the Hilbert space adjoint of (the weak extension of) $\dbar$ with respect to this inner product, and define the complex Laplacian associated to $\mu$ and $h$ by
\be
\boxop\coloneqq\dbar\dbarstar+\dbarstar\dbar.
\ee
This is an unbounded self-adjoint and nonnegative operator that encapsulates the interaction between $\mu$ and $h$. 
In this paper, we only consider $\boxop$ acting on $(0,1)$-forms. We say that $\boxop$ is \emph{$b^2$-coercive} ($b>0$) if  $\boxop\geq b^2$ in the sense of quadratic forms. We refer to \cref{sec:coerc} for precise definitions. We are finally in a position to state our main result.
\begin{thm}\label{thm-1}
Let $(M,h)$ be a complete Hermitian manifold with (Levi--Civita) Ricci curvature bounded from below. Assume that $\mu=e^{-2\psi}\vol$ satisfies the following properties:
\begin{enumerate}
\item\label{item:thm-1_coercivity} $\boxop$ is $b^2$-coercive for some $b>0$.
\item\label{item:thm-1_bound} $\trace_{\omega_h}(i\partial\dbar\psi) + \tfrac{1}{8}|\theta|_h^2\leq B<+\infty$.
\end{enumerate}
Then the Bergman kernel $K_\mu$ satisfies the following estimate for every $\gamma<\sqrt{2}b$:
\bel\label{exp-decay-2}
	|K_\mu(p,q)| e^{-\psi(p)-\psi(q)}\leq C\frac{e^{-\gamma d(p,q)}}{\sqrt{\vol(p,1)\vol(q,1)}}, \qquad (p,q\in M)
\eel
where $C$ depends only on $\gamma$, $b$, $B$, and the bound on the Ricci curvature.\newline 
Moreover, the coercivity condition \cref{item:thm-1_coercivity} holds if the curvature form $F^{\mu}$ satisfies
\bel\label{suff-cond}
	F^\mu \geq \sigma b^2\omega_h + i\left(\frac{\sigma}{2\sigma -1}\right) T\circ \ol T.
\eel
for some $\sigma > \frac12$. If $T=0$, the conclusion still holds under the condition $F^\mu \geq \frac12 b^2\omega_h$ .
\end{thm}

To compare this result with existing ones in the literature (e.g. \cite{lindholm,schuster--varolin,ma--marinescu}), it is useful to reformulate \cref{suff-cond} in terms of the Chern--Ricci form $\Theta_h$ and $i\partial\bar{\partial} \psi$. Indeed, since $\Theta_h=-i\partial\dbar \log\det(h_{j\ol k})$, it follows that if $\mu = e^{-2\psi}\vol$,  then
\bel\label{Fmu-formula}
	F^\mu = i\partial\dbar\psi +\tfrac{1}{2}\Theta_h.
\eel 
Hence, \cref{suff-cond} is equivalent to
\begin{equation} \label{e:reform}
	i\partial\dbar\psi + \tfrac{1}{2}\Theta_h \geq   \sigma b^2\omega_h + i\left(\frac{\sigma}{2\sigma -1}\right) T\circ \ol T.
\end{equation}

Observe that the assumptions in \cref{thm-1} are much simplified if $h$ is K\"ahler. Indeed, if $(M,h)$ is K\"ahler then $T=0$ and hence \cref{suff-cond} reduces to $i\partial\dbar\psi + \tfrac{1}{2}\Theta_h \geq b^2 \omega_h$. On the other hand, since $\theta=0$ and the Chern--Ricci form is bounded from below, condition \emph{(2)} is implied by the assumption that $F^\mu = i\partial\dbar\psi + \tfrac{1}{2}\Theta_h\leq B\omega_h<+\infty$. We obtain the following \namecref{cor:kahler} which is new already in this special (K\"ahler) case. 

\begin{cor}\label{cor:kahler} Let $(M,h)$ be a complete K\"ahler manifold with Ricci curvature bounded from below. Assume that $\mu$ satisfies
\be\label{e:cor1}
	\tfrac{1}{2}b^2\omega_h\leq F^\mu\leq B\omega_h,
\ee 
for some $b>0$ and $B<+\infty$, then \cref{exp-decay-2} holds for $\gamma<\sqrt{2}b$. 
\end{cor}

After a preprint of this work was made public, the authors were informed by Shoo Seto that a result similar to \cref{cor:kahler}, in the special case of polarized K\"ahler manifolds, was obtained in collaboration with Lu, and appeared in his thesis \cite{seto-thesis}.\newline
It is worth noticing that under the assumptions of \cref{cor:kahler}, $F^\mu$ is the fundamental form of a metric $h_\mu$ that is ``comparable'' to $h$, and estimate~\cref{exp-decay-2} also holds with respect to $h_\mu$ (with a possibly different constant $\gamma$). \newline
Also note that when $h$ is the flat metric on $\CC^n$, the condition~\cref{e:cor1} is equivalent to \cref{lindholm}, which is considered by Lindholm \cite{lindholm}. \newline
For the Bergman metric on the unit ball, $\Theta_h=-2(n+1)\omega_h$, and thus \cref{e:cor1} reduces to $i\partial\dbar\psi \geq (n+1+b^2) \omega_h$, which is stronger than the assumption in Theorem~1.1 of \cite{schuster--varolin}.\newline
If $(M,h)$ has bounded geometry in the sense of \cite{ma--marinescu}, then \cref{e:reform} holds if $\psi$ is replaced by $k^2\psi$ for $k$ large enough, provided that $i\partial\dbar\psi \geq  \epsilon \omega_h$ for some $\epsilon>0$, and hence the estimate hold for $\mu^{(k)}\coloneqq e^{-2k^2\psi}\vol$. In \cref{cor3} below, we state precisely the geometric conditions for \cref{exp-decay-2} to hold for $\mu^{(k)}$.

\begin{cor}\label{cor3}
Let $(M,h)$ be a complete Hermitian manifold with Ricci curvature bounded from below. Suppose that there exist $\eta > 1, Q\geq 0$, and $P\in\RR$ such that
\begin{equation}\label{e:bound-torsion-ricci}
	|\theta|^2_h \leq Q, \quad\text{and}\quad
	\Theta_h - i \eta\, T\circ \ol{T} \geq  P \omega_h.
\end{equation}
Suppose further that $\psi$ satisfies 
\bel\label{cond-cor}
\frac12 b^2\omega_h\leq i\partial\dbar\psi\leq B\omega_h, 
\eel 
where $b>0$ and $B<+\infty$. Put $\mu^{(k)}\coloneqq e^{-2k^2\psi}\vol$, $k>0$. If $\gamma<b\sqrt{2(\eta -1)/\eta}$ and $k$ is large enough (depending on $\gamma$, $P$), then the Bergman kernel $K_{\mu^{(k)}}$ satisfies the following estimate:
\bel\label{exp-decay-3}
	|K_{\mu^{(k)}}(p,q)| e^{-k^2\psi(p)-k^2\psi(q)}
	\leq
	C\frac{k^{2n} e^{-\gamma k d(p,q)}}{\sqrt{\vol(p,1)\vol(q,1)}} \qquad (p,q\in M),
	%C\frac{e^{-\gamma k d(p,q)}}{\sqrt{\vol(p,k^{-1})\vol(q,k^{-1})}} \qquad (p,q\in M),
\eel
where $C$ depends only on $\gamma$, $b$, $B$, $Q$, and $K$.
\end{cor}
It is clear that if $h$ is K\"ahler, then the inequalities in \cref{e:bound-torsion-ricci} trivially hold for $\eta>0$ arbitrary large, so that \cref{exp-decay-3} holds for any $\gamma<\sqrt{2}b$ (if $k$ is large enough).

\subsection{Structure of the paper and main ingredients of the proof}
After discussing some generalities about Bergman spaces and complex Laplacians in Section \ref{sec:bergman}, we start to present the ingredients of the proof of \cref{thm-1}.\newline 
As a first step, we establish the following exponential decay of canonical solutions of the $\dbar$-equation, which could be of independent interest.

\begin{thm}\label{thm-2}
Let $(M,h)$ be a complete Hermitian manifold and assume that the smooth positive measure $\mu$ is such that $\boxop$ is $b^2$-coercive for some $b>0$. Let $u\in L^2_{0,1}(M,h,\mu)$ be supported on the geodesic ball $B(p,R)$ and $\dbar$-closed (i.e., assume that $\dbar u=0$) and put $f\coloneqq\dbarstar\boxop^{-1}u$.\newline
Then for every $q \in M$ and $\gamma<2\sqrt{2}b$, the following bound holds:\bel\label{L2-decay}
\int_{B(q,R)}|f|^2\,d\mu\leq C e^{-\gamma d(p,q)}\int_{B(p,R)}|u|^2_h\,d\mu.
\eel
where $C$ depends only on $\gamma$, $b$, and $R$.\newline  
If in addition $(M,h)$ has Ricci curvature bounded from below by $K$ with $K\leq 0$, and $\mu=e^{-2\psi}\vol$ satisfies the condition
\bel\label{cond-torsion}
	\trace_{\omega_h}(\Cplxi \partial\dbar\psi) + \tfrac{1}{8} |\theta|_h^2\leq B<+\infty, 
\eel
then we have the pointwise bound ($\gamma<2\sqrt{2}b$ as above)
\bel\label{pointwise-decay}
|f(q)|^2e^{-2\psi(q)}\leq \frac{C}{\vol(q,R)}e^{-\gamma d(p,q)}\int_{B(p,R)}|u|^2_h\,d\mu, 
\eel where $C$ depends only on $\gamma, b, BR^2$, and $R\sqrt{-K}$.
\end{thm}

Notice that the function $f$ of the statement is the solution of the equation $\dbar f=u$ with minimal $L^2(M,\mu)$ norm (see \cref{sec:coerc} below), that is, the so-called canonical solution.\newline 
The first half of \cref{thm-2} states that, under the sole geometric assumption of completeness of $(M,h)$, coercivity of $\boxop$ implies the $L^2$ exponential decay \cref{L2-decay} of $\dbar^*\boxop^{-1}u$ off the support of $u$. Its proof occupies \cref{sec:dbar} and is based on a method developed by Agmon to establish exponential decay of eigenfunctions of Schr\"odinger operators (see, e.g., \cite{agmon}). The key observation is that $\boxop$ satisfies a \emph{localization formula} analogous to the simple yet very effective IMS localization formula of Schr\"odinger operators (see \cref{sec:loc}). 

In a second step, accomplished in \cref{sec:pointwise}, we improve the $L^2$ decay to an $L^\infty$ decay, exploiting a mean value inequality for nonnegative subsolutions of the heat equation on Riemannian manifolds due to Li and Tam~\cite{li--tam} (but see also \cite{li--schoen}), which holds under a lower bound on the Ricci curvature. To apply this inequality in the Hermitian context, we need to control the difference between the Laplacian of the background Riemannian metric and the Laplacian of the Chern connection, which may be expressed in terms of the torsion and ultimately leads to condition~\cref{cond-torsion}. Thanks to this mean value inequality, we can avoid the ``Kerzman trick'' (as in \cite{delin} and \cite{dallara}) and the ``pluriharmonic recentering of the weight'' techniques (as in \cite{schuster--varolin}). These methods are difficult to implement on manifolds without some sort of ``bounded geometry'' assumptions.

The analysis just sketched has a conditional nature, resting on the assumption that $\boxop$ is coercive (condition \emph{(1)} in \cref{thm-1}). This hypothesis is made more transparent by a ``basic identity with torsion term'' (\cref{prop:basic}), thanks to which we can give a sufficient condition for coercivity that involves only the geometry of the Hermitian metric and the curvature form of the measure (inequality~\cref{suff-cond}). As evidence of the interest of basic identities involving a torsion term, we show in an Appendix that the ``twisted basic identities'' of the kind discussed, e.g., in Section 3 of \cite{mcneal--varolin}, can be thought of as ``standard'' basic identities with respect to conformally K\"ahler metrics. 

The last two sections of the paper (\cref{sec:conclusion} and \cref{sec:ach}) contain the deduction of \cref{thm-1} and \cref{cor3} from \cref{thm-2}, and a discussion of the interesting example of asymptotically complex hyperbolic metrics of Bergman-type.

As a final remark, let us point out that Bergman kernels can be fruitfully defined in the more general setting where holomorphic functions are replaced by holomorphic sections of a holomorphic line bundle on $M$ endowed with a Hermitian metric (see \cite{ma--marinescu/book} for a comprehensive treatment of this matter). Most of our techniques work in this more general framework, but we confine ourselves to the scalar setting for the sake of simplicity. 

\subsection{Further directions}

While the pointwise condition~\cref{suff-cond} is easy to check and sufficient to prove some interesting results, coercivity of $\boxop$ is expected to hold under much weaker conditions (cf. \cite{dallara/C2}). This is mainly due to the fact that, in loose terms, $\boxop$ is a generalized Schr\"odinger operator, as made apparent by the basic identity of \cref{prp:basic}. Condition~\cref{suff-cond} is morally a uniform positive lower bound on the ``potential'' of $\boxop$, while coercivity amounts to positivity of the minimal eigenvalue: in the case of ordinary Schr\"odinger operators it is well-known that the latter condition is much weaker (see, e.g., \cite{mazya--shubin}). This idea have an antecedent in \cite{christ} and is considered in \cite{dallara/C2}, but, to the authors' knowledge, has never been explored in the general context of Hermitian manifolds (but see Theorem~3 of \cite{devyver} for a Riemannian counterpart). 

We also believe that a better analytical understanding of the quadratic form of $\boxop$ would allow an improvement of \cref{cor:kahler} in the same vein as the result of \cite{schuster--varolin} for the unit ball (see the comment after \cref{cor:kahler}).

\subsection{Acknowledgements} The authors would like to thank Jeffery McNeal for an interesting discussion about the subject of this paper. 

\section{Preliminaries on Bergman kernels and the complex Laplacian on Hermitian manifolds with measure}
\subsection{Bergman spaces and Bergman kernels}\label{sec:bergman}
We recall that in the rather general setting of a complex manifold $M$ equipped with a positive Borel measure $\mu$, one may consider the \emph{Bergman space}
\be \bergman(M,\mu) \coloneqq \bigg\{f \colon M \to \Cplx : f \text{ is holomorphic and } \int_M |f|^2\,d\mu < \infty \bigg\},\ee
which is a linear subspace of $L^2(M,\mu)$. While in complete generality this is not the case, for many kind of measures the evaluation maps $f\mapsto f(p)$ are locally uniformly bounded linear functionals on $A^2(M,\mu)$, i.e., for every compact $K \subseteq M$ there is $C(K) <+\infty$ such that
\bel\label{admissible} |f(p)|^2 \leq C(K) \int_M |f|^2\,d\mu \qquad\forall f \in \bergman(M,\mu),\quad \forall p\in K.\eel
This condition is sometimes called \emph{admissibility} of the measure $\mu$ (see, e.g., \cite{pasternak-winiarski} and \cite{zeytuncu}). In this paper we restrict our attention to \emph{smooth positive} measures, that is, measures having smooth positive density with respect to Lebesgue measure in local coordinates. It is a simple consequence of the mean value property of holomorphic functions (in local holomorphic coordinates) that such measures always satisfy the admissibility condition~\cref{admissible}. In any case, under assumption \cref{admissible}, the Bergman space is closed in $L^2(M,\mu)$, so that the associated orthogonal projector \be B_\mu:L^2(M,\mu)\rightarrow A^2(M,\mu),\ee is well-defined, and in fact $A^2(M,\mu)$ is a \emph{reproducing kernel Hilbert space}. Explicitly, there is a function 
\be K_\mu \colon M \times M \to \Cplx,\ee which we call the \emph{Bergman kernel}, that satisfies the following properties:
\begin{enumerate}
 \item $K_\mu(\cdot,q)\in A^2(M,\mu)$ for every $q\in M$ and
  \be\label{eq:berman_kernel_variational}
	\|K_\mu(\cdot,q)\|_{L^2(\mu)}^2=K_\mu(q,q) =  \sup\big\{|f(q)|^2 : f \in \bergman(M,\mu) \text{ and } \|f\| \leq 1\big\};
	\ee
	\item\label{item:bergman_kernel_props_symmetry}
	$K_\mu(p,q) = \ol{K_\mu(q,p)}$;
	\item\label{item:bergman_kernel_props_projection}
	$K_\mu$ is the integral kernel of $B_\mu$:
	\be
	 B_\mu f(p) = \int_M K_\mu(p,q) f(q)\,d\mu(q). \qquad \forall f \in L^2(M,\mu),\quad\forall p \in M.
	\ee
\end{enumerate}
Moreover, the following Cauchy--Schwarz type inequality holds:
\begin{equation}
\label{cauchy-schwarz}
|K_\mu(p,q)| \leq \int_M |K_\mu(p,p')K_\mu(p',q)|\,d\mu(p') \leq \sqrt{K_\mu(p,p)\, K_\mu(q,q)}\qquad\forall p,q\in M.
\end{equation}
For proofs of these properties, see for instance \cite{bergman,pasternak-winiarski}.
%
%As is well-known, the above objects can be fruitfully defined in the more general setting of holomorphic sections of a holomorphic vector bundle on $M$ endowed with a Hermitian metric (see \cite{ma--marinescu/book} for a comprehensive treatment of this matter). Most of our techniques work in this more general framework, but we confine ourselves to the scalar setting for the sake of simplicity.\TODO{This was already mentioned earlier\dots}

\subsection{The complex Laplacian $\boxop$ and its coercivity}\label{sec:coerc}
%
%Let $(M,h)$ be a Hermitian manifold of complex dimension $n$, and suppose $\mu$ is a smooth positive measure on $M$ (we point out that most of the facts discussed below hold under much weaker regularity assumptions).

A Hermitian manifold is a complex manifold $M$ endowed with a Hermitian metric \( h=h_{j\ol k}\,dz^j\otimes dz^{\ol k},\) where $h_{j\ol k}=h_{\ol k j}=\ol{h_{\ol j k}}$. The associated real $(1,1)$-form \(\omega_h\coloneqq ih_{j\ol k}\, dz^j\wedge dz^{\ol k}\) is called the fundamental form (or K\"ahler form) of $h$. As usual, we refer to both $h$ and $\omega_h$ as a metric on $M$. A Hermitian scalar product $\langle\cdot,\cdot\rangle_h$ is induced in the usual way on cotangent spaces: in particular, if $u=u_{\ol j}\, dz^{\ol j}$ and $v=v_{\ol j}\,dz^{\ol j}$ are $(0,1)$-forms, then $\langle u,v\rangle_h=h^{\ol j k}u_{\ol j} v_k$, where $[h^{\ol j k}]$ is the inverse matrix of $[h_{j\ol k}]$, and $v_k\coloneqq\ol{v_{\ol k}}$. This Hermitian scalar product can be extended to tensors of all ranks and our convention for the case of covariant tensors of rank $2$ is that
\bel\label{2-forms}
\langle\eta_1\otimes\eta_2,\eta_3\otimes\eta_4\rangle_h = \frac{1}{2}\langle\eta_1,\eta_3\rangle_h\langle\eta_2,\eta_4\rangle_h,
\eel
whenever the $\eta_k$'s are $1$-forms. The associated norms will be denoted by $|\cdot|_h$.

We identify differential forms with alternating tensors in such a way that $\eta_1\wedge\eta_2 \coloneqq \eta_1\otimes\eta_2-\eta_2\otimes\eta_1$, when $\eta_1$ and $\eta_2$ are $1$-forms. With this definition, if $u$ and $v$ are $(0,1)$-forms, we have
\bel\label{induced-metric}
|u\wedge v|_h^2 + |\langle u,v\rangle_h|^2 = |u|_h^2|v|_h^2.
\eel 

Suppose $\mu$ is a smooth positive measure on $M$ (we point out that most of the facts discussed below hold under much weaker regularity assumptions). We can define $L^2_{(0,q)}(M,h,\mu)$ as the Hilbert space of square-integrable $(0,q)$-forms with respect to $\mu$ and~$h$. More explicitly, if $u$ and $v$ are $(0,q)$-forms, the scalar product on $L^2_{(0,q)}(M,h,\mu)$ has the expression $\int_M\langle u,v\rangle_h\,d\mu$ anticipated in \cref{sec:intro}. We restrict our attention to $q\leq 2$ (recall convention~\cref{2-forms}). Observe that $L^2_{(0,0)}(M,h,\mu) = L^2(M,\mu)$ is the standard $L^2$-space of $\CC$-valued functions, defined with respect to the measure $\mu$.\newline
We define 
\be
\dom_q(\dbar)\coloneqq\big\{u\in L^2_{(0,q)}(M,h,\mu) : \dbar u \in L^2_{(0,q+1)}(M,h,\mu)\big\}, 
\ee 
where the $\dbar$ in the formula above is to be taken in the sense of distributions (or, more precisely, currents). It is clear that $\dbar$ defines an unbounded operator mapping $L^2_{(0,q)}(M,h,\mu)$ into $L^2_{(0,q+1)}(M,h,\mu)$, whose domain is $\dom_q(\dbar)$. This is called the \emph{weak extension} of the differential operator $\dbar$. We skip any reference in the notation to the degree of forms on which $\dbar$ acts, since this should always be clear from the context. Putting all the operators together, we get a \emph{weighted $\dbar$-complex on $(M,h)$}:
\bel\label{dbar-complex}
L^2(M,\mu)\stackrel{\dbar}\longrightarrow L^2_{(0,1)}(M,h,\mu)\stackrel{\dbar}\longrightarrow L^2_{(0,2)}(M,h,\mu)\stackrel{\dbar}\longrightarrow \cdots
\eel
Notice that the operators above are closed, so that \cref{dbar-complex} is a \emph{Hilbert complex} in the sense of \cite{bruning--lesch} (closure follows immediately from the fact that convergence in $L^2_{(0,q)}(M,h,\mu)$ implies convergence in the sense of currents). Thus, we have the dual complex \be
L^2(M,\mu)\stackrel{\dbar^*_{h,\mu}}\longleftarrow L^2_{(0,1)}(M,h,\mu)\stackrel{\dbar^*_{h,\mu}}\longleftarrow L^2_{(0,2)}(M,h,\mu)\stackrel{\dbar^*_{h,\mu}}\longleftarrow \cdots,
\ee where every $\dbar^*_{h,\mu}$ is the Hilbert space adjoint of the corresponding $\dbar$. We decided to use the slightly cumbersome notation $\dbar^*_{h,\mu}$ to stress the fact that not only the domains, but also the ``formulas'' of these first-order differential operators depend on the metric $h$ and the measure~$\mu$. \newline
We are finally in a position to define the complex Laplacian:
\be
\boxop^{(q)}\coloneqq\dbar\dbar^*_{h,\mu} + \dbar^*_{h,\mu}\dbar \qquad (1\leq q\leq n-1).
\ee
The operator $\boxop^{(q)}$ is self-adjoint and nonnegative when considered on the natural domain
\be
\dom(\boxop^{(q)})\coloneqq \big\{u\in \dom_q(\dbar)\cap \dom_q(\dbar^*_{h,\mu})\colon \dbar u \in \dom_{q+1}(\dbar^*_{h,\mu}),\ \dbar^*_{h,\mu} u \in \dom_{q-1}(\dbar)\big\},
\ee
where we used the obvious notation for the domains of the $\dbar^*_{h,\mu}$'s. One can analogously define $\boxop^{(0)} =\dbar^*_{h,\mu}\dbar$ and $\boxop^{(n)} =\dbar\dbar^*_{h,\mu}$. For the purposes of this paper, it is enough to consider the complex Laplacian for $q=1$, and we will consequently drop the superscript, putting $\boxop\coloneqq\boxop^{(1)}$. As usual, a key role is played by the quadratic form 
\be\label{e:energy}
\mathcal{E}_{h,\mu}(u,v)\coloneqq\int_M\langle\dbar u,\dbar v\rangle_h\,d\mu + \int_M\dbar^*_{h,\mu} u\cdot \overline{\dbar^*_{h,\mu} v}\,d\mu,
\ee
which is well-defined whenever $u,v\in \dom_1(\dbar)\cap \dom_1(\dbar^*_{h,\mu})\eqqcolon\dom(\mathcal{E}_{h,\mu})$. Notice that $\dbar^*_{h,\mu} u$ is a scalar function, while $\dbar u$ is a $(0,2)$-form. We adopt the convention that $\mathcal{E}_{h,\mu}(u)\coloneqq\mathcal{E}_{h,\mu}(u,u)$. By definition,
\be
	\mathcal{E}_{h,\mu}(u,v) = \int_M\langle\boxop u,v\rangle_h \, d\mu
\ee
if $u\in \dom(\boxop)$ and $v\in \dom(\mathcal{E}_{h,\mu})$.

Our first restriction on the metric $h$ is justified by the following proposition.
\begin{prop}\label{prp:core}
	If the Hermitian metric $h$ is complete, the space $\mathcal{D}_{(0,1)}$ of smooth compactly supported $(0,1)$-forms is dense in $\dom(\mathcal{E}_{h,\mu})$ with respect to the graph norm. It is also a core of $\boxop$, and the restriction of $\boxop$ to $\mathcal D_{(0,1)}$ is essentially self-adjoint.
\end{prop}

\begin{proof}
 See for instance \cite{ma--marinescu/book}.
 The fact that we do not use the measure induced by the Hermitian metric is of no consequence, since we may rewrite $\mu = e^{-2\psi}\vol$ and view $\mathcal E_{h,\mu}$ as the quadratic form of the complex Laplacian on $(M,h,\vol)$ for forms with values in the trivial line bundle on $M$, with fiber metric given by $e^{-2\psi}$.
\end{proof}

We say that $\boxop$ is \emph{$c$-coercive ($c>0$)} if $\boxop\geq c$ in the sense of quadratic forms or, equivalently, if 
\bel\label{coercivity}
	\energ(u)\geq c\int_M|u|_h^2\,d\mu\qquad\forall u\in \dom(\energ).
\eel
In view of \cref{prp:core}, it is enough that the inequality above holds for $u\in \mathcal{D}_{(0,1)}$. By standard functional analysis, whenever $\boxop$ is $c$-coercive there exists a bounded inverse $\boxop^{-1}$ with domain $L^2_{(0,1)}(M,h,\mu)$ and range $\dom(\boxop)$. The operator norm of $\boxop^{-1}$ is bounded from above by $c^{-1}$. Moreover, under assumption \cref{coercivity}, the $\dbar$-equation \be
\dbar f=u
\ee admits a unique solution orthogonal to the Bergman space $A^2(M,\mu)$, whenever the datum $u$ is in $L^2_{(0,1)}(M,h,\mu)$ and $\dbar$-closed, i.e., $\dbar u=0$. This solution may be expressed as
\be
f=\dbarstar\boxop^{-1}u,
\ee
and satisfies the bound \begin{equation}
\int_M|f|^2d\mu\leq c^{-1}\int_M|u|_h^2d\mu.\label{can_sol_bound}
\end{equation}
For our purposes, the most important consequence of this formula is the well-known \emph{Kohn's identity} for the Bergman projection:
\be\label{eq:kohns_identity}
B_\mu(f) = f-\dbarstar\boxop^{-1}\dbar f\qquad \text{for all } f\in\dom_0(\dbar).
\ee
Notice that while the terms appearing on the right hand side of this identity depend on the metric $h$, the left hand side depends only on $\mu$. It is this asymmetry that gives us the freedom to choose, given $\mu$, the most appropriate metric, e.g., one that makes $\boxop$ coercive (if it exists).\newline 
See, e.g., \cite{straube,berger} for proofs of the well-known facts just discussed.

\section{$L^2$ exponential decay of canonical solutions of the $\bar{\partial}$-equation}\label{sec:dbar}

The goal of this section is to prove the first half of \cref{thm-2}, that is, \cref{coercexp-thm} below. In order to do that, we need a localization lemma and a Caccioppoli-type inequality.

\subsection{A localization formula for $\boxop$}\label{sec:loc}

\cref{local-lem} below is a \emph{localization formula} for $\boxop$ that is analogous to the very useful {IMS localization formula} in the theory of Schr\"odinger operators. For the latter see, e.g., Lemma~3.1 of \cite{simon} or Lemma~11.3 of \cite{teschl}. Before stating and proving it, we need a few preliminaries. 

First, notice that if $\Lip(M,h)$ is the class of scalar functions $\chi\colon M\rightarrow \RR$ that are Lipschitz with respect to the Riemannian distance, then by Rademacher's theorem, $\chi$ is almost everywhere differentiable and 
\be
	|\dbar \chi|_h^2=\tfrac{1}{2}|d\chi|_h^2\leq \tfrac{1}{2}L^2,
\ee 
where $L$ is the Lipschitz constant of $\chi$.

Next, we state the Leibniz rule for $\dbarstar$ for future reference. For this, we employ the notation $w\vee v$ for the \emph{interior product} of the forms $v$ and $w$ (with respect to $h$). This is the form defined by the condition\be
\langle w\vee v, u\rangle_h = \langle v, \overline{w}\wedge u\rangle_h,
\ee where $u$ is an arbitrary form. Observe that the conjugation on the right hand side makes the interior product bilinear. 

\begin{lem}
 If $\chi\in \Lip(M,h)\cap L^\infty(M)$ and $v\in \dom_q(\dbar^*_{h,\mu})$ ($1\leq q\leq n$), then $\chi v\in \dom_q(\dbar^*_{h,\mu})$ and 
\bel\label{leibniz}\dbar^*_{h,\mu}(\chi v)=\chi \dbar^*_{h,\mu} v - \partial \chi\vee v.\eel
\end{lem}

\begin{proof}
Let $u\in \dom_{q-1}(\dbar)$. Then $\dbar(\chi u)=\chi\dbar u + \ol{\partial\chi}\wedge u$ and the remark we made about the differentiability of Lipschitz functions implies immediately that $\chi u\in \dom_{q-1}(\dbar)$. Hence,\be
\int_M\langle u,\chi\dbarstar v\rangle_h\,d\mu = \int_M\langle\dbar(\chi u),v\rangle_h\,d\mu = \int_M\langle\dbar u,\chi v\rangle_h\,d\mu+\int_M\langle u,\partial\chi\vee v\rangle_h\,d\mu,
\ee which gives the thesis.
\end{proof}

\begin{lem}[Localization formula]\label{local-lem}
If $u\in\dom(\boxop)$ and $\chi\in\Lip(M,h)\cap L^\infty(M)$, then $\chi u\in\dom(\mathcal{E}_{h,\mu})$ and the following identity holds:
\be\label{locform}
\mathcal{E}_{h,\mu}(\chi u)=\Re\int_M\langle \boxop u,\chi^2 u\rangle_{h}\,d\mu + \int_M|\dbar\chi|_h^2 |u|_h^2\,d\mu.
\ee
\end{lem}

\begin{proof}
Exactly as in the proof of Lemma 3.1 of \cite{simon}, we compute in two ways the iterated commutator $[\chi,[\chi,\boxop]]$, where $\chi$ is identified with a multiplication operator. We will use \cref{leibniz} a few times without comment. All the computations below are for $u$ and $\chi$ smooth and compactly supported, the statement then follows appealing to \cref{prp:core}. We have
\begin{align}
	\big[\chi,\dbar\dbar^*_{h,\mu}\big] u 
	&= \chi\dbar\dbar^*_{h,\mu}u - \dbar\dbar^*_{h,\mu}(\chi u) \notag \\
	&= \chi\dbar\dbar^*_{h,\mu}u - \dbar(\chi \dbar^*_{h,\mu} u-\partial \chi\vee u) \notag \\
	&=\dbar(\partial \chi\vee u) - \dbar\chi\wedge\dbar^*_{h,\mu}u.
\end{align}
Thus, 
\be
	\big[\chi,\big[\chi,\dbar\dbar^*_{h,\mu}\big]\big]u
	= -2\dbar\chi\wedge\big(\partial\chi\vee\dbar^*_{h,\mu}u\big).
\ee
Analogously, we get
\begin{align}
	\big[\chi,\dbar^*_{h,\mu}\dbar\big] u 
	& = \chi \dbar^*_{h,\mu}\dbar u-\dbar^*_{h,\mu}\dbar(\chi u)\notag \\
	& = \chi \dbar^*_{h,\mu}\dbar u-\dbar^*_{h,\mu}(\dbar\chi\wedge u + \chi\dbar u) \notag \\
	& = -\dbar^*_{h,\mu}(\dbar\chi\wedge u) + \partial\chi\vee\dbar u,
\end{align}
and
\be
	\big[\chi,\big[\chi,\dbar^*_{h,\mu}\dbar\big]\big]u =-2\partial \chi\vee(\dbar\chi\wedge u).
\ee 
Putting everything together, we get, for all $ u\in \dom(\boxop)$,
\be
	-\frac{1}{2}\left[\chi,[\chi,\boxop]\right]u = \dbar\chi\wedge\big(\partial\chi\vee\dbar^*_{h,\mu}u\big)+\partial \chi\vee(\dbar\chi\wedge u).
\ee
On the other hand, we can easily see that
\be
	-\frac{1}{2}[\chi,[\chi,\boxop]]u = \chi\boxop(\chi u ) - \frac{\chi^2\boxop u + \boxop(\chi^2 u)}{2}.
\ee 
Combining the two identities we get
\begin{align}
	\mathcal{E}_{h,\mu}(\chi u)
	& =\int_M\langle \boxop (\chi u),\chi u\rangle_{h}\,d\mu \notag \\
	& = \Re\int_M\langle \boxop u,\chi^2 u\rangle_{h}\,d\mu + \int_M\left(|\partial\chi\vee u|_h^2 + |\dbar\chi\wedge u|_h^2\right)d\mu.
\end{align}
Then \cref{locform} follows by observing that $\partial\chi\vee u = \langle u,\dbar \chi\rangle_h$ and recalling \cref{induced-metric}.
\end{proof} 

\subsection{Caccioppoli-type inequality for $\boxop$-harmonic $(0,1)$-forms}

\begin{prop}\label{cacciop-lem}
Let $u\in\dom(\boxop)$ be such that $\boxop u=0$ on a geodesic ball $B(p,R)$. Then, for every $R'<R$, \be
\int_{B(p,R')} |\dbar^*_{h,\mu}u|^2\,d\mu\leq \frac{2}{(R-R')^2}\int_{B(p,R)}|u|_h^2\,d\mu.
\ee
\end{prop}

\begin{proof}
We define \be
\chi(q)\coloneqq\max\{1-(R-R')^{-1}d(q,B(p,R')),0\},
\ee
where $d$ and $B$ are the geodesic distance and balls associated to $h$, respectively. It is easy to see that $\chi\in \text{Lip}(M,h)\cap L^\infty(M)$ and that $\chi(q)>0$ holds exactly on $B(p,R)$, and that $|\dbar \chi|^2_h\leq (R-R')^{-2}/2$.

Applying the localization formula~\cref{locform} to $\chi u$ one immediately gets
\be
	\int_M |\dbar^*_{h,\mu}(\chi u)|^2\,d\mu\leq \mathcal{E}_{h,\mu}(\chi u)=\int_M|\dbar\chi|_h^2|u|_h^2\,d\mu.
\ee 
Recalling the Leibniz rule~\cref{leibniz}, this gives \be
\int_M |\dbar^*_{h,\mu}u|^2\chi^2\,d\mu\leq 4\int_M|\dbar\chi|_h^2|u|_h^2\,d\mu\leq \frac{2}{(R-R')^2}\int_M|u|_h^2\,d\mu.
\ee Since $\chi\equiv1$ on $B(p,R')$, we are done.
\end{proof}

\subsection{Coercivity implies $L^2$ exponential decay of canonical solutions}
 
\begin{thm}\label{coercexp-thm}
 Assume that $\boxop$ is $b^2$-coercive for some $b>0$, i.e., that \cref{coercivity} holds.
%\bel\label{coerc}
%	\energ(u)\geq b^2\int_M|u|_h^2\,d\mu\qquad\forall u\in\mathrm{dom}(\energ).
%\eel
Then for every $\gamma<2\sqrt{2}b$ and $R >0$, there exists a constant $C_{\gamma,R,b}$ such that if $u\in L^2_{0,1}(M,h,\mu)$ is supported on $B(p,R)$ and $f\coloneqq\dbarstar\boxop^{-1}u$, then
\be\label{e:l2decaycan}
	\int_{B(q,R)}|f|^2\,d\mu\leq C_{\gamma,R,b} \, e^{-\gamma d(p,q)}\int_{B(p,R)}|u|^2_h\,d\mu
\ee
holds for every $q \in M$.
\end{thm}

\begin{proof} By inequality \eqref{can_sol_bound} in \cref{sec:coerc}, under the coercivity condition \cref{coercivity} we have
\be
	\int_M |f|^2 \,d\mu \leq b^{-2} \int_{M}|u|^2_h\,d\mu.
\ee 	
In particular, \cref{e:l2decaycan} holds for $d(p,q) \leq 4R$ with $C_{\gamma,R,b} \geq e^{4\gamma R} \widetilde{C}$.
Thus, without loss of generality, we may assume that $d(p,q)\geq 4R$.\newline 
Since $u$ is supported on $B(p,R)$, we see that $\boxop^{-1}u$ is $\boxop$-harmonic on $B(q,2R)$. Thus, using \cref{cacciop-lem}, we obtain
\bel\label{coerc-1}
	\int_{B(q,R)}|f|^2\,d\mu\leq 2R^{-2}\int_{B(q,2R)}|\boxop^{-1}u|_h^2\,d\mu.
\eel
We introduce the functions
\be
	\widetilde{d}(p')\coloneqq\min\big\{d(p,p'),d(p,q)\big\},
\ee
and 
\be
	\chi(p')\coloneqq\min\big\{1,R^{-1}d(p',B(p,R))\big\}
\ee
for $p' \in M$.
Notice that $\widetilde{d}, \chi\in \lip$, so that we also have $\chi e^{a\widetilde{d}}\in\lip$ with $a>0$. Since $\boxop^{-1}u\in\dom(\boxop)$, we can apply \cref{local-lem} to get
\be
	\energ(\chi e^{a\widetilde{d}} \, \boxop^{-1}u)=\Re\int_M\langle u,\chi^2e^{2a\widetilde{d}}\,\boxop^{-1}u\rangle_{h}\,d\mu + \int_M|\dbar(\chi e^{a\widetilde{d}})|_h^2 \, |\boxop^{-1}u|_h^2\,d\mu.
\ee
Observe that $\chi$ was chosen to be $0$ on the support of $u$, and hence the first term on the right hand side vanishes. Recalling the coercivity condition~\cref{coercivity}, we obtain
\be
	b\sqrt{\int_M\chi^2e^{2a\widetilde{d}}\,|\boxop^{-1}u|_h^2\,d\mu}\leq \sqrt{\int_M|\dbar\chi|_h^2 \, e^{2a\widetilde{d}} \, |\boxop^{-1}u|_h^2\,d\mu} + a\sqrt{\int_M|\dbar\widetilde{d}|_h^2 \, \chi^2e^{2a\widetilde{d}}\, |\boxop^{-1}u|_h^2\,d\mu}.
\ee
The pointwise bound $|\dbar \widetilde{d}|_h^2\leq 1/2$ suggests that we choose $a<\sqrt{2}b$ and reabsorb the rightmost term. By support considerations and the bound $|\dbar\chi|_h^2\leq R^{-2}/2$, we finally get\be
\left(b-\frac{a}{\sqrt{2}}\right)^2\int_{B(q,2R)}e^{2a\widetilde{d}}|\boxop^{-1}u|_h^2\,d\mu\leq \frac{1}{2R^{2}}\int_{B(p,2R)}e^{2a\widetilde{d}} |\boxop^{-1}u|_h^2\,d\mu.
\ee
Our choice of $\widetilde{d}$ guarantees that this function is bounded from below by $d(p,q)-2R$ on $B(q,2R)$, and from above by $2R$ on $B(p,2R)$. Thus,
\be\label{e:boxin}
	\left(b-\frac{a}{\sqrt{2}}\right)^2\int_{B(q,2R)}|\boxop^{-1}u|_h^2\,d\mu\leq \frac{e^{8aR}}{2R^{2}}e^{-2ad(p,q)}\int_{B(p,2R)}|\boxop^{-1}u|_h^2\,d\mu.
\ee
To complete the proof we combine \cref{e:boxin,coerc-1}, and the observation that $\boxop^{-1}$ is bounded with operator norm at most $b^{-2}$, so that we have $\int_{B(p,2R)} |\boxop^{-1}u|_h^2\,d\mu \leq \int_M |\boxop^{-1}u|_h^2\,d\mu \leq b^{-4} \int_{B(p,R)} |u|_h^2\,d\mu$.
\end{proof}

\section{From $L^2$ to pointwise bounds}\label{sec:pointwise}
The key ingredient in the transition to pointwise bounds from the $L^2$-bounds of \cref{coercexp-thm} is the following result by Li--Schoen and Li--Tam.
\begin{thm}%[Li--Schoen \cite{li--schoen}, Li--Tam \cite{li--tam}] 
\label{thm:litam}
	Let $(M,g)$ be a complete Riemannian manifold, $p\in M$ and $R>0 $ be such that the geodesic ball $B(p,2R)$ does not meet the boundary of $M$. Suppose that the Ricci curvature of $g$ is bounded below by $K$ with $K\leq 0$. Let $\delta \in (0,\frac{1}{2})$, $q > 0$, and $\lambda \geq 0$. Then there exists a constant $C$ that depends only on $\delta, q, \lambda R^2$, and $R\sqrt{-K}$ such that for any nonnegative smooth function $f$ on $B(p,2R)$ satisfying the differential inequality
	\begin{equation}
	\Delta_g f \geq -\lambda f
	\end{equation}
	we have
	\begin{equation}
	\sup_{B(p, (1-\delta)R)} f^q \leq \frac{C}{\vol(p,R)}\int _{B(p,R)} f^{q} \,\dvol_g.
	\end{equation}
\end{thm}

This is essentially Corollary 3.6 of \cite{hga}, which follows easily from the results on subsolutions of the heat equation on Riemannian manifolds of \cite{li--tam}.

Here, the convention is that $\Delta_g$ is nonpositive. To apply this theorem, we will need to
compare the Riemannian Laplacian $\Delta_g f$, where $g\coloneqq2\Re h$, with the so-called Chern Laplacian $\trace_{\omega_h}(\Cplxi\partial\dbar f)$  of a regular function~$f$.
The comparison is well-known and is stated and proved in \cref{torsion-form-prp} below for convenience (cf. formula~(25) in \cite{gauduchon}).

%We denote by $\widetilde{\nabla}$ the Levi-Civita connection of the Riemannian metric $g\coloneqq2\Re h$. Since $h$ is Hermitian, the Christoffel symbols $\widetilde{\Gamma}^k_{\bar{i}j}$ of the Levi-Civita connection in local holomorphic coordinates reduce to
%\begin{equation}\label{e:lc}
%%	\widetilde{\Gamma}^k_{ij}
%%	=
%%	\frac{1}{2}h^{k\bar{l}} \left(\partial_i h_{j\bar{l}} + \partial_{j} h_{i\bar{l}}\right),
%%	\quad
%	\widetilde{\Gamma}^k_{\bar{i}j}
%	=
%	\frac{1}{2}h^{k\bar{\ell}} \left(\partial_{\overline{i}} h_{j\bar{\ell}} - \partial_{\overline{\ell}} h_{j\bar{i}}\right).
%\end{equation}
%The symbols for other combinations of barred and unbarred indices can of course be deduced from these.
%The associated (non-positive) Laplace-Beltrami operator $\Delta_g$, i.e., the trace of the Hessian $\widetilde{\nabla}^2$, is given by the formula 
%\be
%	\Delta_g f=2 h^{j\bar{k}}\bigl(\widetilde{\nabla} df\bigr)_{j\bar{k}}.
%\ee
%
%
% We need the following well-known elementary identity . We give a proof for the sake of completeness.

\begin{prop}\label{torsion-form-prp} For a smooth function $f$ on the Hermitian manifold $(M,h)$, one has 
\be
\Delta_gf= 2\trace_{\omega_h}(\Cplxi\partial\dbar f)  +\langle d f,\theta\rangle_h,
\ee
where $\theta$ is the torsion 1-form defined by \eqref{torsion-1-form}.
\end{prop}

\begin{proof} Let $\widetilde{\nabla}$ denote the Levi--Civita connection of $g\coloneqq2\Re h$. Since $h$ is Hermitian, the Christoffel symbols $\widetilde{\Gamma}^k_{\bar{i}j}$ in local holomorphic coordinates reduce to
	\begin{equation}\label{e:lc}
	%	\widetilde{\Gamma}^k_{ij}
	%	=
	%	\frac{1}{2}h^{k\bar{l}} \left(\partial_i h_{j\bar{l}} + \partial_{j} h_{i\bar{l}}\right),
	%	\quad
	\widetilde{\Gamma}^k_{\bar{i}j}
	=
	\frac{1}{2}h^{k\bar{\ell}} \left(\partial_{\overline{i}} h_{j\bar{\ell}} - \partial_{\overline{\ell}} h_{j\bar{i}}\right).
	\end{equation}
It then follows that
\begin{equation}
	\widetilde{\Gamma}^k_{\bar{i} k}
	=
	\frac{1}{2} h^{k\bar{\ell}} \left(\partial_{\overline{i}} h_{k\bar{\ell}} - \partial_{\overline{\ell}} h_{k\bar{i}}\right)
	=
	\frac{1}{2} T_{\bar i \bar{\ell}}^{\bar \ell},
\end{equation}
where $T_{\bar i \bar{\ell}}^{\bar \ell}$ is the torsion (0,1)-form of the Chern connection. Thus
\begin{equation}
	h^{j\bar{k}}  \widetilde{ \Gamma}^{\bar{\ell}}_{j\bar{k}} = - \widetilde{\Gamma}^{\bar{k}}_{i\bar{k}} h^{i\bar{\ell}} = - \frac{1}{2} h^{i\bar{\ell}} T^k_{ik}.
\end{equation}
Locally, $\trace_{\omega_h}(i\partial\bar{\partial} f) = h^{j\bar{k}} \partial_j \partial_{\bar{k}} f$ and therefore
\begin{align}
	\Delta_g f
	&= 
	2 h^{j\bar{k}}\bigl(\widetilde{\nabla} df\bigr)_{j\bar{k}} \notag \\
	&= 
	2 h^{j\bar{k}} \left[\partial_j \partial_{\bar{k}} f -   \widetilde{\Gamma}^{\bar{l}}_{j\bar{k}} \partial_{\bar{l}}f - \widetilde{\Gamma}^{l}_{j\bar{k}}
	 \partial_{l} f \right] \notag \\
	&= 
	2\trace_{\omega_h}(\Cplxi\partial\dbar f)+
	\langle d f,\theta\rangle_h. \qedhere
\end{align}
\end{proof}
As a consequence of \cref{thm:litam,torsion-form-prp}, we have the following mean
value inequality. Recall that $\psi$ was defined to satisfy $\mu = e^{-2\psi}\,\text{Vol}$. 

\begin{lem}\label{mean-value-lem} Assume that the Ricci curvature is bounded from below by $K\leq0$ on $B(p,2R)$ and put
\be
	\lambda\coloneqq\sup_{B(p,2R)}\left\{\trace_{\omega_h}(\Cplxi\partial\dbar \psi) + \tfrac{1}{8}|\theta|_h^2\right\},
\ee
where $\theta$ is the torsion 1-form. If $F\colon B(p,2R)\rightarrow\CC$ is holomorphic, then
\be\label{e:meanvalue}
	|F(p)|^2 e^{-2\psi(p)}\leq \frac{C}{\vol(p,R)}\int_{B(p,R)}|F|^2\,d\mu,
\ee 
where the constant $C$ depends only on $\lambda R^2$ and $R\sqrt{-K}$.
\end{lem}

\begin{proof}
Let $f\coloneqq|F|^2e^{-2\psi}$. First, observe that by the Cauchy--Schwarz inequality 
\begin{align}
	\left|\langle df,\theta\rangle_h\right | 
	& =
	2|\langle(\dbar \ol F- 2\ol F\dbar\psi),\ol F\theta\rangle_h|e^{-2\psi} \notag \\
	& \leq
	2 |\partial F - 2F\partial\psi|_h^2e^{-2\psi} +  \frac{1}{2}f|\theta|_h^2.
\end{align}
Next, we compute 
\begin{align}
\trace_{\omega_h}(\Cplxi\partial\dbar f) 
& =
h^{j\overline{k}}\partial_j\partial_{\overline{k}}f \notag \\
& = h^{j\overline{k}}(\partial_jF - 2\partial_j\psi F)\overline{(\partial_kF - 2\partial_k\psi F)}e^{-2\psi}
-
2 h^{j\overline{k}}(\partial_j\partial_{\overline{k}}\psi) e^{-2\psi} \notag \\
& =\left(|\partial F-2F\partial\psi|_h^2 - 2|F|^2 \trace_{\omega_h}(\Cplxi\partial\dbar \psi)\right)e^{-2\psi} \notag \\
& = 
|\partial F-2F\partial\psi|_h^2 e^{-2\psi} - 2\trace_{\omega_h}(\Cplxi\partial\dbar \psi) f.
\end{align}
Putting the two estimates together and exploiting \cref{torsion-form-prp}, we obtain, on $B(p,2R)$,
\begin{align}
	\Delta_gf
	& \geq 
	2 \trace_{\omega_h}(\Cplxi\partial\dbar f) 
	-
	\left|\langle df,\theta\rangle_h\right | \notag  \\
	& \geq
	-4\left(\trace_{\omega_h}(\Cplxi\partial\dbar \psi) + \frac{1}{8} |\theta|_h^2\right)f \notag \\
	& \geq
	-4\lambda f.
\end{align}
This estimate, together with the lower bound on the Ricci curvature, shows that the hypothesis of \cref{thm:litam} are satisfied. Thus, 
\be
	f(p)\leq \frac{C}{\vol (p,R)}\int_{B(p,R)}f\,\dvol,
\ee 
where $C$ depends on $\lambda R^2$ and $R\sqrt{-K}$, as we wanted.
\end{proof}

Combining \cref{mean-value-lem,coercexp-thm}, we immediately obtain:
\begin{thm}\label{can-point-thm}
Let $(M,h)$ be a complete Hermitian manifold with Ricci curvature bounded from below by $-K$ ($K\leq 0$). Let $\mu = e^{-2\psi} \vol$ be a smooth positive measure such that $\boxop$ is $b^2$-coercive. Suppose further that
\begin{equation}
	\trace_{\omega_h}(i \partial\bar{\partial} \psi) + \tfrac{1}{8}|\theta|^2 \leq B < +\infty.
\end{equation}
Let $u\in L^2_{0,1}(M,h,\mu)$ be supported on $B(p,R)$ and $\dbar$-closed, and put $f\coloneqq\dbarstar\boxop^{-1}u$. For every $q \in M$ and $\gamma<2\sqrt{2}b$, the following bound holds:
\be
	|f(q)|^2 e^{-2\psi(q)}\leq \frac{C}{\vol(q,R)}e^{-\gamma d(p,q)}\int_{B(p,R)}|u|^2_h\,d\mu, 
\ee 
where $C$ depends on $\gamma, b, BR^2$, and $R\sqrt{-K}$.
\end{thm}

This completes the proof of \cref{thm-2}.

\section{The basic identity for $\boxop$}\label{sec:basic}
We devote this section to a discussion of the basic identity for $\boxop$, which is essentially \cite{griffiths} applied to $(0,1)$-forms with compact support. We provide a simple proof of this case. 

We denote by $\nabla$ the Chern connection of $h$. In local holomorphic coordinates, the only nonvanishing  Christoffel symbols of $\nabla$ are $\Gamma_{jk}^{\ell}$ and $\Gamma_{\ol j\ol k}^{\ol\ell} = \overline{\Gamma_{jk}^{\ell}}$, where 
\be
\Gamma_{jk}^{\ell}=h^{\ol m \ell} \partial_j h_{k \ol m}, %=-h_{k \ol m}\frac{\partial h^{\ol m \ell}}{\partial z^j} 
%\quad \text{and} \quad \Gamma_{\ol j\ol k}^{\ol\ell} = \overline{\Gamma_{jk}^{\ell}}.
\ee
We shall only need the $(0,1)$-part of $\nabla$, which we denote by $\ol \nabla$. In particular, if $u=u_{\ol k}\,dz^{\ol k}$, $\ol \nabla u$ is the 2-tensor 
\bel\label{chern-conn}
\ol \nabla u = \left (\partial_{\ol j} u_{\ol k} - \Gamma^{\ol \ell}_{\ol j\ol k}u_{\ol \ell}\right)dz^{\ol j}\otimes dz^{\ol k}.
\eel
The key to our proof of the basic inequality is an elementary pointwise identity that involves only the metric $h$. In order to state it, we recall the standard notation $u^\sharp$ for the vector field associated to the $1$-form $u$ by the metric $h$. Notice that if $u$ is a $(0,1)$-form, then $u^\sharp$ is a $(1,0)$-vector field, and  $\ol\nabla u^\sharp$ is a $2$-tensor with one covariant and one contravariant index.

\begin{lem}\label{lem:torsion} For every $(0,1)$-form $u = u_{\bar{i}} dz^{\bar{i}}$, the following identity holds:
\be
 \trace\big(\ol \nabla u^\sharp\otimes \nabla \ol u^\sharp\big)=2|\ol \nabla u|_h^2 - |\dbar u-Tu|_h^2,
\ee
where $\trace\big(\ol \nabla u^\sharp\otimes \nabla\ol u^\sharp\big) = (\ol\nabla u^\sharp)_{\ol j}^{m}(\nabla \ol u^\sharp)^{\ol j}_{m}$ and $Tu = T^{\bar{i}}_{\bar{j}\bar{k}} u_{\bar{i}} dz^{\bar{j}} \otimes dz^{\bar{k}}$.

\end{lem}

\begin{proof}
Notice that in local coordinates $u^\sharp=u^m\,\partial_m$ where $u^m\coloneqq h^{m\ol k}u_{\ol k}$, and recall that one of the defining properties of the Chern connection is that\be
 \ol \nabla_{\partial_{\ol j}}(u^\sharp) = \partial_{\ol j}u^{m}\,\partial_m.
 \ee It is thus clear that the trace in the statement is \bel\label{trace}
 \partial_{\ol j}u^{m}\, \partial_mu^{\ol j} = h^{m\ol k}\big(\partial_{\ol j}u_{\ol k}-\Gamma^{\ol p}_{\ol j\ol k}u_{\ol p}\big)\, h^{\ol j\ell}\big(\partial_{m}u_\ell-\Gamma^{q}_{m\ell}u_q\big).
 \eel
Now notice that if $A=A_{\ol j \ol k}\,dz^{\ol j}\otimes dz^{\ol k}$ and $\widetilde{A}=A_{\ol k\ol j }\,dz^{\ol j}\otimes dz^{\ol k}$, a straightforward computation gives \be
2|A|_h^2-|A-\widetilde{A}|_h^2 = A_{\ol j\ol k}A_{m\ell}h^{\ol j\ell}h^{\ol k m}.
\ee
If $A=\ol\nabla u$, by \cref{chern-conn} we have $A-\widetilde{A}=\dbar u-Tu$, and the identity above becomes\be
2|\ol \nabla u|_h^2 - |\dbar u-Tu|_h^2= \big(\partial_{\ol j}u_{\ol k}-\Gamma^{\ol p}_{\ol j\ol k}u_{\ol p}\big)\big(\partial_{m}u_\ell-\Gamma^{q}_{m\ell}u_q\big) h^{\ol j\ell}h^{\ol km}.\ee
In view of \cref{trace}, this is the formula we set out to prove.
\end{proof}
\begin{prop}\label{prop:basic}
 \label{prp:basic}
 Let $(M,h)$ be a Hermitian manifold and $\mu$ a positive smooth measure with curvature form $F^{\mu}$.
 Then for every $u \in \mathcal{D}_{(0,1)}(M)$,
 \begin{equation}
  \label{eq:basic_identity}
\int_M\left|\dbar u - Tu\right|_h^2\,d\mu + \int_M|\dbarstar u|^2\,d\mu = 2\int_M |\ol \nabla u|_h^2\,d\mu + 2\int_M \langle F^\mu,\overline{u}\wedge u\rangle_h\,d\mu
 \end{equation}
and, for any $\nu >0$,
 \begin{equation}\label{basic-inequality}
 \energ(u) \geq %\int_M|\dbarstar u|^2\,d\mu +
 \frac{2}{1+\nu}\int_M |\ol \nabla u|_h^2\,d\mu 
 +
 \frac{2}{1+\nu}\int_M \langle F^\mu,\overline{u}\wedge u\rangle_h\,d\mu 
 - 
 \frac{1}{\nu}\int_M|Tu|_h^2\,d\mu
 \end{equation}
where $\energ(u)$ is defined by \cref{e:energy} and $Tu$ is defined in \cref{lem:torsion}.%\TODO{We should define $Tu$. DONE}
\end{prop}

It may be of interest to remark that the right hand side of \eqref{eq:basic_identity} has the following explicit epression in terms of the $(1,0)$-vector field $u^{\sharp}=u^\ell\partial_\ell$ (and the metric $h$):\begin{equation*}
2\int_M h_{\ell\overline m}h^{\overline j k} \partial_{\overline j}u^\ell  \partial_ku^{\overline m}d\mu + 2\int_M \partial_\ell\partial_{\overline m}\varphi u^\ell u^{\overline m}d\mu.
\end{equation*}
If $\dim M=1$, the expression above does not contain explicitly the metric, making the analysis of $\Box_{h,\mu}$ much simpler.

\begin{proof}
It is enough to prove the identity for $u$ supported on a coordinate chart with coordinates $z^j$. Let $\varphi$ be the real-valued function such that $d\mu=\Cplxi e^{-2\varphi}\, dz^1\wedge dz^{\ol 1}\wedge\dotsb\wedge dz^n\wedge dz^{\ol n}$. Then the adjoint of $\partial_m$ with respect to $d\mu$ is $\delta_m\coloneqq-\partial_m +2\partial_m\varphi$. Integrating both sides of the identity of \cref{lem:torsion}, the usual commutation argument yields\be
2\int_M|\ol \nabla u|_h^2\,d\mu-\int_M|\dbar u-Tu|_h^2\,d\mu= \int_M|\delta_mu^m|^2\,d\mu - 2\int_M\partial_m\partial_{\ol j}\varphi\, u^m u^{\ol j}\,d\mu.
 \ee
To complete the proof of \cref{eq:basic_identity}, one may easily check that $\partial_m\partial_{\ol j}\varphi\, u^m u^{\ol j}=\langle F^\mu,\overline{u}\wedge u\rangle_h$ and that $\delta_mu^m = \dbarstar u$. The basic inequality \cref{basic-inequality} follows immediately.
\end{proof}
\begin{cor}\label{cor:pointwisecondition}
	Let $(M,h)$ be a complete Hermitian manifold and $\mu$ a smooth positive measure on $M$. Suppose that
	\begin{equation}
		F^{\mu} \geq \sigma b^2 \omega_h + i \left(\frac{\sigma}{2\sigma -1}\right) T\circ \ol{T}, \quad b>0 \text{ and } \sigma > \frac12.
	\end{equation}
	Then the associated complex Laplacian $\boxop$ is $b^2$-coercive. If $T=0$, then the conclusion still holds under the assumption $F^\mu \geq \frac12 b^2\omega_h$.
\end{cor}
\begin{proof}
	Observe that $\langle i T \circ \ol{T} , \ol u \wedge u\rangle_h = |Tu|^2$. The statement then follows from \cref{basic-inequality} with $\nu = 2\sigma -1$.
\end{proof}

\section{Proof of \cref{thm-1,cor3}}\label{sec:conclusion}
We first prove \cref{thm-1}. By the hypothesis, \cref{mean-value-lem} may be applied uniformly on geodesic balls of radius $1$. In particular, \bel\label{mean}
|F(p)|^2\lesssim \frac{e^{2\psi(p)}}{\vol(p,1)}\int_{B(p,1)}|F|^2\,d\mu\qquad\forall F\in A^2(\mu).
\eel
where the implicit constant depends only on $B$ and the $K$.
By \cref{eq:berman_kernel_variational}, we have
\bel\label{diag}
|K_\mu(p,p)|\lesssim \frac{e^{2\psi(p)}}{\text{Vol}(p,1)},
\eel 
and therefore, by the Cauchy--Schwarz inequality~\cref{cauchy-schwarz},
\begin{align}
	|K_\mu(p,q)|^2e^{-2\psi(p)-2\psi(q)}
	& \leq
	|K_\mu(p,p)||K_\mu(q,q)|e^{-2\psi(p)-2\psi(q)} \notag \\
	& \lesssim 
	\frac{1}{\vol(p,1)\vol(q,1)}.
\end{align}
To prove the off-diagonal exponential decay \cref{exp-decay-2}, we can assume without loss of generality that $d(p,q)\geq 4$. Applying \cref{mean} to $K_\mu(\cdot,q)$ we get:
\be
	|K_\mu(p,q)|^2\lesssim \frac{e^{2\psi(p)}}{\vol(p,1)}\int_{B(p,1)}|K_\mu(p',q)|^2\,d\mu(p').
\ee
Observe that if $\chi_p(p')\coloneqq\max\{0,1-d(B(p,1),p')\}$, the definition of $B_\mu$ and Kohn's identity~\cref{eq:kohns_identity} give
\begin{align}
	\int_{B(p,1)}|K_\mu(p',q)|^2\,d\mu(p') 
	& \leq
	\int_MK_\mu(q,p')K_\mu(p',q)\chi_p(p')\,d\mu(p') \notag \\
	& = 
	B_\mu(K_\mu(\cdot,q)\chi_p)(q) \notag \\
	& = 
	K_\mu(q,q)\chi_p(q) - (\dbarstar\boxop^{-1}u)(q),
\end{align}
where $u=\dbar\left(K_\mu(\cdot,q)\chi_p\right)=K_\mu(\cdot,q)\dbar\chi_p$ is a $(0,1)$-form supported on $B(p,2)$ (in fact on $B(p,2)\setminus B(p,1)$). Since $d(p,q)\geq4$, the first term vanishes. The second one may be bounded with \cref{can-point-thm}, which gives
\be
	\left|(\dbarstar\boxop^{-1}u)(q)\right|\lesssim\frac{e^{\psi(q)}}{\sqrt{\vol(q,2)}}e^{-\gamma d(p,q)}\sqrt{\int_{B(p,2)}|K_\mu(\cdot,q)|^2|\dbar\chi_p|_h^2\,d\mu},
\ee 
where $\gamma<\sqrt{2}b$. Notice that by taking the square root we lose a factor $2$, with respect to \cref{can-point-thm}. Since $|\dbar\chi_p|_h^2\leq1/2$, we finally estimate
\begin{align}
	\int_{B(p,2)}|K_\mu(p',q)|^2|\dbar\chi_p(p')|_h^2\,d\mu(p')
	&\lesssim
	|K_\mu(q,q)|\int_{B(p,2)}|K_\mu(p',p')|\,d\mu(p') \notag \\
	&\lesssim
	\frac{e^{2\psi(q)}}{\vol(q,1)}\int_{B(p,2)}\frac{\dvol(p')}{\vol(p',1)},
\end{align}
where we used the diagonal bound \cref{diag}. By the Bishop--Gromov volume comparison theorem \cite[Theorem 4.19]{gallot--hulin--lafontaine}, $\vol(p',1)\approx\vol(p,2)$ for every $p'\in B(p,2)$, where the implicit constant depends just on the dimension of $M$ and the lower bound on the Ricci curvature. Thus
 \begin{equation}
	 \int_{B(p,2)}\frac{\dvol(p')}{\vol(p',1)}\approx 1.
 \end{equation}
Analogously, \(\vol(q,2)\approx \vol(q,1) \), and therefore 
\be
	|(\dbarstar\boxop^{-1}u)(q)|\lesssim\frac{e^{2\psi(q)}}{\vol(q,1)}e^{-\gamma d(p,q)},
\ee 
which, together with the estimates obtained above, gives the desired estimate. The last part of the theorem follows from \cref{cor:pointwisecondition}. 

We now turn to the proof of \cref{cor3}. By assumption, we can take $b'<b$ such that $\gamma < 2b'\sqrt{(\eta -1)/\eta}$. Choose $k$ large enough such that $k^2(b^2 - b'^2) + P \geq 0$ so that 
	\begin{equation}
	k^2(b^2 - b'^2) \omega_{h} \geq  i\eta\,  T_h \circ \ol{T}_h -\Theta_h. 
	\end{equation}
	Putting $h^{(k)}: = k^2 h$, by direct calculations we get
	\begin{align}
	F^{\mu(k)} & =
	k^2 i\partial \bar{\partial} \psi + \frac{1}{2}\Theta_h \notag \\
	& \geq \frac{1}{2}k^2 b'^2 \omega_h + \left(\frac{\eta}{2}\right) i T_h \circ \ol{T}_h \notag  \\
	& = \frac{1}{2} b'^2 \omega_{h^{(k)}} + \left(\frac{\eta}{2}\right) i T_{h^{(k)}} \circ \ol{T}_{h^{(k)}}.
	\end{align}
	By \cref{cor:pointwisecondition}, the complex Laplacian $\Box_{h^{(k)}, \mu^{(k)}}$ is $\widetilde{b}^2$-coercive with $\widetilde{b}^2 = b'^2(\eta -1)/\eta$ and \cref{thm-1} holds with $\gamma < b'\sqrt{2(\eta -1)/\eta} = 2\widetilde{b}$. On the other hand,
	\begin{equation}
	\mu^{(k)} = e^{-2k^2\psi} \vol_h
	=
	e^{-2k^2\psi} k^{-2n}\vol_{h^{(k)}}
	=
	e^{-2\psi^{(k)}}\vol_{h^{(k)}}
	\end{equation}
	with $\psi^{(k)} = k^2 \psi + 2n \log k$. Therefore,
	\begin{equation}
	\trace_{\omega_{h^{(k)}}} (i\partial \bar{\partial} \psi^{(k)}) + \tfrac{1}{8}|\theta|_{h^{(k)}}^2
	\leq
	nB + \tfrac{1}{8} k^{-2} Q.
	\end{equation}
	Applying \cref{thm-1} for $\mu^{(k)}$ and $h^{(k)}$, we obtain for $p,q \in M$,
	\begin{align}
	|K_{\mu^{(k)}}(p,q)| e^{-k^2\psi(p)-k^2\psi(q)}
	&\leq C\frac{e^{-\gamma d_{h^{(k)}}(p,q)}}{\sqrt{\vol_{h^{(k)}}(p,1)\vol_{h^{(k)}}(q,1)}}\notag \\
	&= C\frac{e^{-\gamma k d_h(p,q)}}{\sqrt{\vol_h(p,k^{-1})\vol_{h}(q,k^{-1})}},
	\end{align}
	where $C$ does not depend on $k$. Finally, observe that by the assumption on the lower bound of the Ricci tensor,
	\[
		\vol_h (p, 1) \leq \vol_h(p, k^{-1}) k^{2n} e^{\sqrt{-K}}, \quad k \geq 1
	\]
	and similarly for $q$. Plugging these into the inequality above, we finally obtain \cref{exp-decay-3}.

\section{An example: ACH metrics of Bergman-type}\label{sec:ach}

In this last section, we discuss in some detail an interesting example. Let $D$ be a precompact strictly pseudoconvex domain in a complex manifold $X$ with smooth boundary. Suppose that $D$ is defined by $\varrho < 0$, with $d\varrho \ne 0$ on $\partial D$ and $\varrho$ is smooth in a neighborhood $U$ of $\partial D$. We further assume that $-\log (-\varrho)$ is strictly plurisubharmonic on $U\cap D$. In this case, $ -i\partial \bar{\partial} \log (-\varrho)$ defines an asymptotically complex hyperbolic (ACH) K\"ahler metric $h_{\varrho}$ on $U\cap D$.\newline
Given any Hermitian metric $\widetilde{h}$ on $D$, we can patch, using a partition of unity, $\widetilde{h}$ and $h_{\varrho}$ to obtain a Hermitian metric $h$ on $D$ such that $h =h_{\varrho}$	on $U\cap D$. It is known that the curvature tensor of $h$ approaches the curvature tensor of constant holomorphic sectional curvature $-4$ (see \cite{klembeck}). In particular, the sectional curvature is bounded from above, while $\operatorname{Ric}_h$ and $\Theta_h$ are bounded from below. The last fact is easy to see: near the boundary $\partial D$, in local coordinates
 
	\begin{equation*}
		\operatorname{Ric}_h
		=
		-(n+1) \omega_{\phi} - i \partial\bar{\partial} \log J[\varrho]
	\end{equation*}
		where $J[\varrho]$ is the (Levi--)Fefferman determinant:
	\begin{equation*}
		J[\varrho] = - \det 
		\begin{bmatrix}
		\varrho & \varrho_{\bar{j}} \\
		\varrho_k & \varrho_{k\bar{j}}
		\end{bmatrix}.
	\end{equation*}
Notice that $i \partial\bar{\partial} \log J[\varrho]$, which does not depend on the local coordinates, extends smoothly to a neighborhood of $\partial D$, and is hence bounded. Moreover, since $T_h=0$ near the boundary, $h$ must have bounded torsion.\newline 
Also note that in general the metric $h$ constructed in this way is non-K\"ahler and need not have bounded geometry.\newline 		
Suppose that $\mu$ is a smooth measure on $D$ such that $\boxop$ is $b^2$-coercive and with $\Delta_{h} \log (\dvol_h/d\mu)$ bounded from above. Then the Bergman kernel $K_{\mu}$ satisfies the exponential decay estimate \cref{exp-decay-2}, namely
		\begin{equation*}
		|K_{\mu}(p,q)| \leq \frac{C}{\eta(p)\eta(q)} e^{-\gamma d_{h}(p,q)},
		\end{equation*}
where $\eta = \dvol_h/d\mu$, $d_h$ is the Riemannian distance of $h$, and $\epsilon$ depends on the coercivity constant $b$. Observe that the volume factors have been absorbed into the constant since $h$ has sectional curvature bounded from above.\newline 
Moreover, by \cref{cor3}, if $i\partial\bar{\partial} \eta > \epsilon \omega_h$ for some $\epsilon >0$, then for $k$ large enough
\begin{equation*}
|K_{\eta^{k^2} d\vol_h}(p,q)| \leq \frac{Ck^{2n}}{\eta^k(p)\eta^k(q)} e^{-\gamma k d(p,q)},
\end{equation*}
for some constants $C$ and $\gamma$ do not depend on $k$.\newline 
When $h= -i\partial \bar{\partial} \log (-\varrho)$ is defined on $D$, the coercivity is satisfied under an explicit condition on $\mu$ and $\varrho$, i.e, when
	\begin{equation*}
		i \partial\bar{\partial} \log \left(d\mathrm{vol}_h/d\mu\right) \geq i \partial\bar{\partial} \log ((-\varrho)^{-n-1-b} J[\varrho]).
	\end{equation*}
	If $D\subset \mathbb{C}^n$ and $d\mu = e^{-\varphi}d\lambda$, then the condition $F^\mu \geq b\omega_\phi$ translates into
	\begin{equation*}
			i \partial\bar{\partial} \log [(-\varrho)^b e^\varphi] \geq 0,
	\end{equation*}
	in other words, when $\log [(-\varrho)^b e^\varphi]$ is strictly plurisubharmonic for some positive constant~$b$.

\section*{Appendix}

We now show that the ``twisted basic identities'' of the kind discussed in, e.g., \cite{mcneal--varolin} are particular instances of Proposition \ref{prop:basic} when the metric is conformally K\"ahler.\newline 
To see this, let $(M,h)$ be a K\"ahler manifold and $\tau:M\rightarrow(0,+\infty)$ a smooth ``twisting factor''.  One may easily verify that the diagram below is an isomorphism of Hilbert complexes, i.e., that the vertical arrows are unitary isomorphisms and the two squares commute:

\begin{center}\begin{tikzpicture}
	\matrix (m) [matrix of math nodes,row sep=3em,column sep=4em,minimum width=2em]
	{
		L^2(M,\mu) &	L^2_{(0,1)}(M,h,\mu) & L^2_{(0,2)}(M,h,\mu) \\
		L^2(M,\tau^{-1}\mu) &	L^2_{(0,1)}(M,\tau^{-1} h,\tau^{-1}\mu) & L^2_{(0,2)}(M,\tau^{-1} h,\tau^{-1}\mu) \\};
	\path[-stealth]
	(m-1-1) edge node [above] {$\dbar\circ\sqrt{\tau}$} (m-1-2)
	(m-1-2) edge node [above] {$\sqrt\tau \circ \dbar$} (m-1-3)
	(m-1-1) edge node [left] {$\sqrt\tau$} (m-2-1)
	(m-1-2) edge node [left] {$1$} (m-2-2)
	(m-1-3) edge node [left] {$1/\sqrt\tau$} (m-2-3)
	(m-2-1) edge node [above] {$\dbar$} (m-2-2)
	(m-2-2) edge node [above] {$ \dbar$} (m-2-3)
  ;
	\end{tikzpicture}
\end{center}
Here the vertical arrows are multiplication operators by the indicated functions. Therefore, the twisted complex in the top row is isomorphic to the weighted $\dbar$-complex in the second row. As a consequence, putting $\widetilde{h}:=\tau^{-1}h$ and $\widetilde{\mu}:=\tau^{-1}\mu$, we have \begin{equation*}
\mathcal{E}_{\widetilde{h}, \widetilde{\mu}}(u)=\int_M\tau|\dbar^*_{h,\mu}u|^2d\mu + \int_M\tau|\dbar u|_h^2d\mu.
\end{equation*}
We now compute the left hand side using \cref{prop:basic}. Let $\eta := -\log \tau$ and observe that the Christoffel symbols of the Chern connection of $\widetilde{h}$ are $\widetilde{\Gamma}^{i}_{jk} = {\Gamma}^{i}_{jk} + \eta_j \delta^i_k$, where ${\Gamma}^{i}_{jk}$ are the Christoffel symbols of the Chern connection of $h$ and $\eta_j$ is a shorthand for $\partial_j\eta$. Hence, we have
$\overline{\nabla}^{\widetilde{h}}_{\bar{j}} u_{\bar{k}}
=
\overline{\nabla}^{h}_{\bar{j}} u_{\bar{k}} - \eta_{\bar{j}} u_{\bar{k}}$.
Representing covariant derivatives with respect to $\nabla^h$ by indices preceeded by a vertical bar $|$, we have
\begin{equation*}
|\overline{\nabla}^{\widetilde{h}} u|^2_h
=
\left(|\overline{\nabla}^{h} u|^2_{h}
+
\frac 12|\bar{\partial} \eta|^2_{h} \, |u|^2_{h}
-
\Re \left(u_{\bar{k}|\bar{j}}\, \eta^{\bar{j}} u^{\bar{k}}\right)\right).
\end{equation*}
Moreover, since the torsion of $\widetilde{h}$ is given by $\widetilde{T}^i_{jk} = \eta_j \delta^i_k - \eta_k \delta^i_j$, we have $\widetilde{T}^i_{jk} u_i = \eta_j u_k - \eta_k u_j$ and
\begin{equation*}
|\widetilde{T}^i_{jk}u_i|^2_h
=
|\bar{\partial} \eta|^2_{h} \, |u|^2_{h}
-
|\langle u , \bar{\partial} \eta \rangle_h|^2.
\end{equation*}
Equation \eqref{eq:basic_identity} yields
\begin{align*}
\mathcal{E}_{\widetilde{h}, \widetilde{\mu}}(u) 
& =
2\int_M |\overline{\nabla}^{\widetilde{h}} u|_{\widetilde{h}}^2d\widetilde{\mu}
+
2 \int_M \langle F^{\widetilde\mu} , \overline{u} \wedge u \rangle _{\widetilde{h}} d\widetilde{\mu} \\ 
& \quad + 2\Re \int_M \langle \bar{\partial} u , \widetilde{T} u \rangle_{\widetilde{h}} d\widetilde{\mu} - \int_M |\widetilde{T} u |_{\widetilde{h}}^2 d\widetilde{\mu}\\
&=
2\int_M \tau |\overline{\nabla}^{h} u|^2_{h}d\mu 
+ 2\int_M  \left\langle \tau F^{\mu}+\frac{i}{2}\partial\dbar\tau, \overline{u} \wedge u \right\rangle_h d\mu \\
& \quad - 2\Re \int_M \tau \left(u_{\bar{k} | \bar{j}}\, \eta^{\bar{j}} u^{\bar{k}}\right) d\mu + 2 \Re \int_M \tau \langle \bar{\partial} u , \widetilde{T} u \rangle_h d\mu.
\end{align*}
We will now use the identity
\begin{equation*}
\langle \bar{\partial} u , \widetilde{T} u \rangle_h 
=
u_{\bar{k} | \bar{j}}\, \eta^{\bar{j}} u^{\bar{k}}
-
u_{\bar{j} | \bar{k}}\, \eta^{\bar{j}} u^{\bar{k}},
\end{equation*}
and the fact that the metric $h$ is K\"ahler. This allows to integrate by parts in the following way:
\begin{align*}
&- 2\Re \int_M \tau \left(u_{\bar{k} | \bar{j}}\, \eta^{\bar{j}} u^{\bar{k}}\right) d\mu + 2 \Re \int_M \tau \langle \bar{\partial} u , \widetilde{T} u \rangle_h d\mu\\
&=- 2\Re  \int_M e^{-\eta } \left(u_{\bar{j} | \bar{k}}\, \eta^{\bar{j}} u^{\bar{k}}\right) d\mu \\
&= 
2\Re \int u_{\bar{j}} \left(\eta^{\bar{j}}u^{\bar{k}} e^{-\eta -2\psi} \right)_{| \bar{k}} d\text{Vol} \notag \\
& = 
2 \int_M \eta_{j\bar{k}} u^{j} u^{\bar{k}} e^{-\eta} d\mu
-
2\Re \int_M \langle u ,\bar{\partial} \eta \rangle _h (\overline{\bar{\partial}^{\ast}_{\mu, h} u}) e^{-\eta} d\mu \notag \\
& \quad - 
2 \int_M  |\langle u , \bar{\partial} \eta \rangle|^2_h  e^{-\eta} d\mu \notag  \\
& = 
-2\int_M \tau_{j\bar{k}} u^j u^{\bar{k}} d\mu + 2\Re \int_M \langle u ,\bar{\partial} \tau \rangle _h (\overline{\bar{\partial}^{\ast}_{\mu, h} u}) d\mu .
\end{align*}
Notice that we used the elementary identity $\bar{\partial}^{\ast}_{\mu, h} u = -\left(u^{\bar{k}} e^{-2\psi} \right)_{| \bar{k}}$. Putting everything together, we find
\begin{align*}
\int_M \tau |\bar{\partial} u|^2_{h}d\mu
+
\int_M \tau |\bar{\partial} ^{\ast} _{\mu,h} u|^2 d\mu
& =
2 \int_M \tau |\overline{\nabla}^{h} u|^2_{h}d\mu 
+ 2\int_M  \tau \left\langle F^{\mu} - \frac{i}{2}\partial\bar{\partial }\tau , \overline{u} \wedge u \right\rangle _h d\mu \notag \\
& \quad +2 \Re \int_M  (\bar{\partial}^{\ast}_{\mu,h} u)\overline{\langle u, \bar{\partial} \tau \rangle}_h d\mu.
\end{align*}
This is Theorem 3.1 of \cite{mcneal--varolin} for $\Omega=M$.

\bibliographystyle{amsalpha}
\bibliography{bergman}
\end{document}

%%% Local Variables:
%%% mode: latex
%%% TeX-master: t
%%% End: